\documentclass{amsart}
\usepackage{amsmath}
\usepackage{amsthm,amssymb,color,comment}
\usepackage{bbm}
\usepackage[TS1,T1]{fontenc}
\usepackage[latin1]{inputenc}
\usepackage{dsfont}
\usepackage{tikz}
\usepackage{enumitem}

\numberwithin{equation}{section}

\newtheorem{thm}{Theorem}[section]

\newtheorem{lem}[thm]{Lemma}

\newtheorem{Def}[thm]{Definition}

\theoremstyle{definition}

\DeclareMathOperator{\DIV}{div}

\newcommand{\R}{\mathbb{R}}

\newcommand{\stress}{\mathbb{S}}
\newcommand{\D}{\mathbb{D}}
\newcommand{\p}{\partial}

\newcommand{\symmetric}{\R^{3\times 3}_{\text{sym}}}

\newcommand{\wstar}{\overset{\ast}{\rightharpoonup}}

\makeatletter
\newcommand{\doublewidetilde}[1]{{%
  \mathpalette\double@widetilde{#1}%
}}
\newcommand{\double@widetilde}[2]{%
  \sbox\z@{$\m@th#1\widetilde{#2}$}%
  \ht\z@=.9\ht\z@
  \widetilde{\box\z@}%
}
\makeatother
\author{Jakub Wo\'{z}nicki}
\address{Faculty of Mathematics, Informatics and Mechanics, University of Warsaw, Stefana Banacha 2, 02-097 Warsaw, Poland}
\email{jw.woznicki@student.uw.edu.pl}
\thanks{J.W. received support from the National Science Centre (Poland), 2017/27/B/ST1/01569.}

\begin{document}

\title{Mv-strong uniqueness for density dependent, non-Newtonian, incompressible fluids}

\begin{abstract}
We consider density dependent, non-Newtonian, incompressible system with the space being flat torus. The viscious stress in the momentum equation is understood through the rheological law and its connection to the proper convex potential. We define the dissipative measure-valued solutions for the aforementioned equations as well as provide a proof of its existence. The main result of this work is the mv-strong uniqueness of the defined solutions.
\end{abstract}


\maketitle

\section{Introduction}

\noindent Recently the measure-valued solutions have experienced a resurgence, as Brenier et al. have proven the weak-strong uniqueness property for measure-valued solutions of incompressible Euler equations in \cite{brenier2011weak}, which were expected to be highly  not unique. The fore-mentioned feature (now usually referred to as mv-strong uniqueness) states that given a strong solution to the system, the weak one necessarily coincides with it. Work of Sheffer \cite{scheffer1993aninviscid} has shown that there is no hope for such a trait even for the distributional solutions of Euler's equations, therefore the main point of \cite{brenier2011weak} is to introduce some kind of energy conservation property from which the single solution would emerge. This "relative energy method" or "relative entropy method" was first used by Dafermos (\cite{dafermos1979thesecond}, \cite{dafermos2016hyperbolic}) in regard to the scalar conservation laws and it is used to describe how two physical systems differ in time, starting from close initial states; its other applications range from stability studies, asymptotic limits (see \cite{christoforou2018relative}, \cite{giesselemann2017stability}, \cite{feireisl2012relative}, \cite{bella2014dimension}) to problems arising from biology (\cite{brezina2018measure}, \cite{demoulini2012weak}, \cite{gwiazda2015weakstrong}, \cite{feireisl2016dissipative}).\\
\\
Language used to describe the measure-valued solutions was introduced by Tartar in \cite{tartar1983thecompensated} or DiPerna in \cite{diperna1985measure}, where they described the solutions with Young measures. Fore-mentioned tool is often used to deal with the weak convergence of approximating sequences. Indeed, the fundamental theorem of Young measures (see for example \cite{rindler2018calcvar}, \cite{ball1989version}) gives a clear result on the weak limits of the sequence $\{f(u_n)\}$ where $f$ is any continuous function (not only a linear one) for bounded sequence $\{u_n\}$. This kind of transformation grants us linearization at the cost of information as we start dealing merely with the measures. Since the classical Young measures only describe oscillations, the framework above may work for the systems where no concentrations appear. To treat this case the so-called generalized Young measures were introduced (\cite{alibert1997nonuniform}), which are the triple of measures describing oscillations, concentrations and concentration angle respectively. Such a way to see this problem was for example used in \cite{brenier2011weak} for the incompressible Euler system.\\
\\
To deal with concentrations also a different approach was introduced in the form of dissipative measure-valued solutions (\cite{feireisl2016dissipative}). Those kinds of solutions consisted of the concentration measure which appeared in the distributional description of the system and the dissipation defect, appearing in the entropy inequality. It differs from the previous considerations in the fact that no measure for the concentration angle is considered. To compensate it, the assumption that the dissipation defect dominates the concentration measure is proposed. Dissipative measure-valued solutions found their use in the recent work by Gwiazda et al. \cite{gwiazda2020dissipative} to describe general conservation laws or for both incompressible and compressible cases of Navier-Stokes together with Euler systems (see an overview of those results in \cite{wiedemann2017fluid}).\\
\\
The main inspiration of this paper is the very recent work by Abbatiello and Feireisl \cite{abbatiello2020onaclass}, in which they prove the existence and the mv-strong uniqueness for non-Newtonian, incompressible fluids. In particular they show the existence of a solution such that its symmetric gradient is merely a measure and the connected concentration measure appears in both momentum equation and energy inequality. The main variation from the aforementioned results is the complete avoidance of the Young measures connected to the oscillations. Only the concentration measure together with dissipation defect are needed to describe the system and the needed bound is achieved with the clever trick in the form of Lemma \ref{tracebound}. We try to generalize this result to the non-homogeneous case, but slightly strengthen the assumptions. The discussion on the reasoning for this is carried out in the Section 1.2 of this paper. The system we are concerned with has been also studied by Wr\'{o}blewska-Kami\'{n}ska in \cite{wroblewska2013unsteady}, where the proof of the existence of the weak-solutions has been provided. The main difference is that the considered N-function connected to the system needed to be supported by the modulus to the power $p \geq \frac{11}{5}$. For the measure-valued solution this constraint might be weakened.\\
\\
The structure of this paper is as follows: in the second section we define the dissipative measure valued solution, in the third section we prove the following existence theorem
\begin{thm}
    Let $u_0\in H(\mathbb{T}^3)$ (where $H$ denotes the closure of the space of smooth, divergence-free functions in $L^2$ norm) and $\rho_0\in L^\infty(\mathbb{T}^3)$ with the bounds
    $$
    0< \rho_*\leq \rho_0\leq \rho^*< \infty
    $$
Then there exists a dissipative measure-valued solution to our system.
\end{thm}
\noindent and the fourth section is devoted to proving the mv-strong uniqueness theorem
\begin{thm}
Let the pair $(\rho, u)$, together with $\stress$ and $m$ be a dissipative measure-valued solution with initial datum $(\rho_0, u_0)$. Then if $P\in C^1([0, T]\times\mathbb{T}^3)$ and $U\in C^1([0, T]\times\mathbb{T}^3; \R^3)$, together with $\widehat{\stress}\in C^1([0, T]\times\mathbb{T}^3; \symmetric)$ is a strong solution with the same initial conditions, then $u = U$, $\rho = P$, $\stress = \widehat{\stress}$ and $m\equiv 0$.
\end{thm}

\subsection{Preliminaries}
The system of equations which we will lay focus on is of the form
\begin{align}\label{equations}
    \p_t{\rho} + \DIV_x(\rho u) = 0\\
    \p_t(\rho u) + \DIV_x(\rho u\otimes u) + \nabla_x p = \DIV_x\stress\label{secondequation}\\
    \DIV_x(u) = 0
\end{align}
where $\rho$ is the mass density, $u$ denotes velocity field, $\stress$ the stress tensor and $p$ is the pressure. Following the work of Abbatiello and Feireisl \cite{abbatiello2020onaclass} we will assume that the viscious stress tensor is related to the symmetrical part of the velocity gradient (denoted $\D u$) through the rheological law
\begin{align}\label{rheolo}
    \stress : \D u = F(\D u) + F^*(\stress)
\end{align}
where
\begin{align}\label{convexfunction}
F:\R^{3\times 3}_{\text{sym}}\rightarrow [0, \infty), \text{ }F(0) = 0, \text{ Dom}(F) = \R^{3\times 3}_{\text{sym}}
\end{align}
is a convex function, $F^*$ is its conjugate, i. e.:
\begin{align}
    F^*(\stress) = \sup_{\D \in \symmetric}(\stress : \D - F(\D))
\end{align}
and $\symmetric$ denotes the space of real, symmetric matrices. Let us recall quickly some of the properties of such a function $F$
\begin{itemize}
    \item it is continuous (from Jensen's inequality it is upper bounded on any simplex and using for example \cite{showalter97monotone} Lemma 7.1 one gets the result),
    \item $F^*$ is lower semi-continuous,
    \item $F^*$ is superlinear i. e.:
    \begin{align}\label{suplin}
    \liminf_{|\stress|\to+\infty}\frac{F^*(\stress)}{|\stress|} = +\infty
    \end{align}
    Indeed for an arbitrary $r > 0$
    $$
    F^*(\stress) \geq \sup_{\D\in \overline{B}(0, r)}(\D : \stress - F(\D)) \geq \sup_{\D\in \overline{B}(0, r)}(\D : \stress - c) \geq \stress : \frac{r\stress}{|\stress|} - c = r|\stress| - c
    $$
\end{itemize}
Additionally, we will assume $F$ to be superlinear
\begin{align}\label{suplinf}
\liminf_{|\D|\to+\infty}\frac{F(\D)}{|\D|} = +\infty
\end{align}

\noindent To avoid any technical difficulties concerning the boundary conditions, we will presume our space to be the flat torus $\mathbb{T}^3$. We denote the initial datum as
\begin{align}
    u(0, x) = u_0(x)\\
    \rho(0, x) = \rho_0(x)
\end{align}
Let us notice that by the virtue of \eqref{convexfunction}
\begin{align}
    \stress\in\p F(\D u)
\end{align}
where $\partial$ denotes the subdifferential of a convex function, or equivalently
$$
\stress : (\D - \D u) \leq F(\D) - F(\D u)
$$ 
for any $\D$.\\
\\
\noindent Our last assumption of technical nature concerns density, namely we assume the existence of two constants $\rho_*$ and $\rho^*$, such that
\begin{align}\label{densityrestraint}
0 < \rho_* \leq \rho_0(x) \leq \rho^* \text{ for almost every }x\in\mathbb{T}^3
\end{align}

\subsection{Remark on the assumptions}
As already mentioned in the introduction the present paper follows in its assumptions \cite{abbatiello2020onaclass}. In their work Abbatiello and Feireisl presuppose that
$$
\liminf_{|\D|\to+\infty}\frac{F(\D)}{|\D|} > 0
$$
The need to strengthen it to \ref{suplinf} comes from the renormalization property of transport equation. Namely, to the knowledge of the author, there is no such result, when the known vector field in the transport equation is only of bounded deformation (which means that its symmetric gradient is a Radon measure, as in \cite{abbatiello2020onaclass}). The result by DiPerna and Lions \cite{diperna1989ordinary} covers the Sobolev spaces case, Ambrosio's in \cite{ambrosio2004transport} the BV functions. Our case, where the symmetric gradient is the $L^1$ function, has been solved in \cite{capuzzo1996onsome}. A similar result is present in \cite{ambrosio2005traces}, where Ambrosio et al. prove the renormalization property for the functions in SBD space with some additional properties. To avoid this kind of change of assumptions one may make use of the Young measure framework hinted at the beginning.

\section{Dissipative measure-valued solution}

\noindent The dissipative measure-valued solutions will be defined in terms of the velocity field $u$, density $\rho$, the stress tensor $\stress$ as well as the additional term connected to possible concentrations: the parameterized family of measures $m$.
\begin{Def}
We say that the pair $(\rho, u)$ is a dissipative measure-valued solution to the problem \eqref{equations} - \eqref{densityrestraint} if 
\begin{align*}
&\rho \in C([0, T]; L^2(\mathbb{T}^3))\cap L^\infty([0, T]\times\mathbb{T}^3)\\
&u\in L^{\infty}((0, T); L^2(\mathbb{T}^3; \R^3))\\
&\D u\in L^1((0, T)\times\mathbb{T}^3; \symmetric),
\end{align*}
there exist
$$
\stress \in L^1((0, T)\times\mathbb{T}^3;\symmetric) 
$$
and a parameterized family of measures $m\in L^\infty((0, T) ;\mathcal{M}(\mathbb{T}^3; \symmetric))$, such that
\begin{itemize}
    \item The bounds on density don't change in time i. e.:
    \begin{align}
        0 < \rho_* \leq \rho(t, x) \leq \rho^*
    \end{align}
    for almost every $(t, x)\in [0, T]\times\mathbb{T}^3$,
    \item Incompressibility\label{incompressibility}
    \begin{align}
        \int_0^T\int_{\mathbb{T}^3}u\cdot\nabla_x\phi\,\mathrm{d}x\mathrm{d}t = 0
    \end{align}
    for every $\phi\in C^1([0, T]\times\mathbb{T}^3)$,
    \item Momentum equation
    \begin{multline}\label{momentumequation}
        \int_0^\tau\int_{\mathbb{T}^3}\rho u\cdot\p_t\phi + \rho u\otimes u:\nabla_x\phi\,\mathrm{d}x\mathrm{d}t - \int_0^\tau\int_{\mathbb{T}^3}\stress:\D\phi\,\mathrm{d}x\mathrm{d}t + \int_0^\tau\int_{\mathbb{T}^3}\nabla_x\phi\,\mathrm{d}m_t\mathrm{d}t\\
        = \int_{\mathbb{T}^3}\rho u\cdot\phi(\tau, x) - \rho_0 u_0\cdot\phi(0, x)\,\mathrm{d}x
    \end{multline}
    for almost every $\tau\in [0, T]$ and divergence-free $\phi\in C^1([0, T]\times\mathbb{T}^3; \R^3)$,
    \item Total mass conservation
    \begin{align}\label{something}
        \int_0^\tau\int_{\mathbb{T}^3}\rho\p_t\phi + \rho u\cdot\nabla_x\phi\,\mathrm{d}x\mathrm{d}t = \int_{\mathbb{T}^3}\rho\phi(\tau, x) - \rho_0\phi(0, x)\,\mathrm{d}x
    \end{align}
    for almost every $\tau\in [0, T]$ and divergence-free $\phi\in C^1([0, T]\times\mathbb{T}^3;\R^3)$,
    \item Energy inequality
    \begin{multline}\label{energyinequality}
        E(\tau) = \frac{1}{2}\int_{\mathbb{T}^3}\rho(\tau,x)|u(\tau, x)|^2\,\mathrm{d}x + \frac{1}{\gamma}\int_{\mathbb{T}^3}\rho(\tau, x)^\gamma\,\mathrm{d}x + D(\tau) + \int_0^\tau\int_{\mathbb{T}^3}F(\D u) + F^*(\stress)\,\mathrm{d}x\mathrm{d}t\\
        \leq \frac{1}{2}\int_{\mathbb{T}^3}\rho_0|u_0(x)|^2\,\mathrm{d}x + \frac{1}{\gamma}\int_{\mathbb{T}^3}\rho_0(x)^\gamma\,\mathrm{d}x
    \end{multline}
    holds for almost every $\tau\in [0, T]$ and some $\gamma > 1$, where
    $$
    D(\tau) = \frac{1}{2}\text{trace}(m_\tau(\mathbb{T}^3))
    $$
    is a dissipation defect,
    \item There exists a constant $C > 0$, such that
    \begin{align}\label{trace}
        |m_\tau|(\mathbb{T}^3) \leq C D(\tau)
    \end{align}
    for almost every $\tau\in [0, T]$.
\end{itemize}
\end{Def}

\section{Existence}

\begin{thm}
    Let $u_0\in H(\mathbb{T}^3)$ (where $H$ denotes the closure of the space of smooth, divergence-free functions in $L^2$ norm) and $\rho_0\in L^\infty(\mathbb{T}^3)$ with the bounds
    $$
    0< \rho_*\leq \rho_0\leq \rho^*< \infty
    $$
Then there exists a dissipative measure-valued solution to our system.
\end{thm}

\noindent The proof is based on the two-step approximation. The first part highly resembles the proof of the Theorem 5.16 \cite{malek1996weak}, p. 267 - 275 and the second part draws inspiration from both \cite{abbatiello2020onaclass} and \cite{wroblewska2013unsteady}.

\subsection{Approximation in $n$ and $\alpha$}
\noindent Let us define
$$
F_{\alpha}(\D) = \inf_{\stress\in\symmetric}{\frac{1}{2\alpha}\|\stress - \D\| + F(\stress)}
$$
which is the Moreau-Yosida approximation of $F$. Due to Moreau's theorem (\cite{showalter97monotone}, Proposition 1.8, p. 162):
\begin{align}
    &F_{\alpha}\text{ converge upwards to }F\text{ as }\alpha\to 0\\
    &\text{The derivative }F_{\alpha}' \text{ is a }\frac{1}{\alpha}-\text{Lipchitz function}
\end{align}
And by the Dini's theorem (\cite{jurgen2005postmodern}, Theorem 12.1, p. 157) we may deduce
\begin{align}\label{uniformconv}
    F_\alpha \nearrow F \text{ unifromly on compact sets }
\end{align}
We can also notice that similarly as for $F$ we have $F_{\alpha}(0) = 0$ and $\text{Dom}(F_\alpha) = \symmetric$. We will use this approximation to get an easier description of the subdifferential (as $F_{\alpha}$ is differentiable). It is also clear from the definition of $F^*$ that since $F_{\alpha} \leq F$ for any $\alpha > 0$ we can infer $F^*\leq F^*_{\alpha}$ for any $\alpha > 0$.  Let us denote by $\{\omega^r\}$ the orthonormal base of $L^2(\mathbb{T}^3;\mathbb{R}^3)$ of smooth, divergence-free functions. We define the approximate problem
\begin{align}\label{galerkin}
    \left\{\begin{array}{ll}
    \p_t\rho^n + \DIV_x(\rho^n\,u^n) = 0\\
    \rho^n(0) = \rho^n_0
    \end{array}\right.
\end{align}
and
\begin{multline}\label{galerkin2}
    \int_{\mathbb{T}^3}\rho^n(t, x)\p_tu^n(t, x)\cdot\omega^r(x)\,\mathrm{d}x + \int_{\mathbb{T}^3}\rho^n(t, x)\,[u^n(t, x)\nabla_xu^n(t, x)]\cdot\omega^r(x)\,\mathrm{d}x\\
    + \int_{\mathbb{T}^3}\stress^n_{\alpha}(t, x) : \D\omega^r(x)\,\mathrm{d}x = 0
\end{multline}
for any $r$, where 
$$
\stress^n_{\alpha}(t, x) = F'_{\alpha}(\D u^n(t,x)) \text{ (the derivative of $F_\alpha$ in the point $\D u^n(t,x)$)}
$$
We will assume that
\begin{align}
    &\rho^n_0 \rightarrow \rho_0\text{ strongly in }L^\infty(\mathbb{T}^3)\\
    & P^n u_0 \rightarrow u_0 \text{ strongly in }H(\mathbb{T}^3)
\end{align}
where $P^n$ is the projection onto first $n$ vectors. To find the approximate solutions we define
$$
B_L(0) = \{a\in C([0, T];\mathbb{R}^n)\,|\,\|a\|_{\infty} \leq L\}
$$
for some $L > 0$. Fix $n\in\mathbb{N}$ and let $\overline{a}\in B_L(0)$ be such that
$$
\overline{a}_k(0) = (u_0, \omega^k)\text{,   }k = 1, 2, ..., n
$$
where $(\cdot,\cdot)$ denotes the scalar product in $L^2$. Set
$$
\overline{u}^n(t, x) = \sum_{k=1}^n\overline{a}_k(t)\omega^k(x)
$$
We will seek the solutions of the system
\begin{align}\label{rhobar}
    \left\{\begin{array}{ll}
    \p_t\rho^n + \DIV_x(\rho^n\,\overline{u}^n) = 0\\
    \rho^n(0) = \rho^n_0
    \end{array}\right.
\end{align}
and
\begin{multline}\label{approximate}
    \int_{\mathbb{T}^3}\rho^n(t, x)\p_tu^n(t, x)\cdot\omega^r(x)\,\mathrm{d}x + \int_{\mathbb{T}^3}\rho^n(t, x)\,[\overline{u}^n(t, x)\nabla_xu^n(t, x)]\cdot\omega^r(x)\,\mathrm{d}x\\
    + \int_{\mathbb{T}^3}\stress^n_{\alpha}(t, x) : \D\omega^r(x)\,\mathrm{d}x = 0
\end{multline}
for
$$
\stress^n_{\alpha}(t, x) = F'_{\alpha}(\D u^n(t,x))
$$
of the form
$$
u^n(t, x) = \sum_{k=1}^na_k(t)\omega^k(x)
$$
We will set the function
\begin{align}\label{functionschauder}
f: \overline{a}\in B_L(0) \longmapsto a
\end{align}
and prove that it fulfills the Schauder fixed-point theorem. Since $\DIV_x(\overline{u}^n) = 0$ we have
\begin{align}
    \left\{\begin{array}{ll}
    \p_t\rho^n + \overline{u}^n\cdot\nabla_x\rho^n = 0\\
    \rho^n(0) = \rho^n_0
    \end{array}\right.
\end{align}
and by the standard characteristics method we obtain a solution $\rho^n \in C^1([0, T]\times \mathbb{T}^3)$. Since
$$
0<\rho_*\leq \rho_0 \leq \rho^*
$$
we can get from the strong convergence that
\begin{align}\label{densitybound}
0 < \rho_* \leq \rho^n(t, x) \leq \rho^*
\end{align}
Having obtained the $\rho^n$ we can dive into the second system. It can be written in the form of the ODE
\begin{align}
    B\cdot\frac{d}{dt}a = h(t, a(t))
\end{align}
where
\begin{align}
    B = [b_{ij}]_{ij} = \left[\int_{\mathbb{T}^3}\rho^n(t, x)\omega^i(x)\cdot\omega^j(x)\,\mathrm{d}x\right]_{ij}
\end{align}
and
$$
h(t, a(t)) = (h_j(t, a(t))_{j=1}^n\text{, }a_0 = (a_{0j})_{j=1}^n
$$
for
\begin{multline*}
    h_j(t, a(t)) = -\int_{\mathbb{T}^3}\rho^n(t, x)\left[\left(\sum_{k=1}^n\overline{a}_k(t)\omega^k(x)\right)\left(\sum_{k=1}^na_k(t)\nabla_x\omega^k(x)\right)\right]\cdot\omega^j(x)\,\mathrm{d}x\\
    - \int_{\mathbb{T}^3}\stress^n_{\alpha}(t,x):\D\omega^j(x)\,\mathrm{d}x
\end{multline*}
\begin{align*}
a_{0j} = (u_0, \omega^j)
\end{align*}
As $B$ is invertible we can write
\begin{align*}
    \left\{\begin{array}{ll}
    \frac{d}{dt}a = B^{-1}\,h(t, a(t))\\
    a(0) = a_0
    \end{array}\right.
\end{align*}
The local solvability of this system is a straight consequence of the Carath\'{e}odory's theorem. For the readers convinience let us recall its formulation
\begin{thm}\textit{(Theorem 3.4, \cite{malek1996weak}, p. 287)}\\
Let F satisfy the Carath\'{e}odory conditions i. e.
\begin{itemize}
    \item $F: (t_0 -\delta, t_0 + \delta)\times \{c\in\mathbb{R}^N\,|\, |c - c_0| < \Delta\}\rightarrow \mathbb{R}^N$, for some $\Delta > 0$,
    \item $F(\cdot, c)$ is a measurable function for all $c\in \{c\in\mathbb{R}^N\,|\, |c - c_0| < \Delta\} = K$,
    \item $F(t, \cdot)$ is a continuous function for almost all $t\in (t_0 -\delta, t_0 + \delta) = I_\delta$,
    \item There exists an integrable function $G:I_\delta \rightarrow \mathbb{R}$ such that
    $$
    |F(t, c)| \leq G(t)
    $$
    for all $(t, c)\in I_\delta\times K$
\end{itemize}
Then there exist $\delta' \in (0, \delta)$ and a continuous function $c: I_{\delta'}\rightarrow \mathbb{R}^N$ such that
\begin{itemize}
    \item $\frac{dc}{dt}$ exists for almost all $t$,
    \item $\frac{d}{dt}c(t) = F(t, c(t))$,
    \item $c(t_0) = c_0$.
\end{itemize}
\end{thm}

\noindent To obtain the global solution we will show that there exists $L$ big enough and $t^*$ small enough such that
\begin{align}
    &\|a\|_{\infty, (0, t^*)} \leq L\label{estimatea}\\
    &\left\|\frac{d}{dt}a\right\|_{\infty, (0, t^*)}\leq K\label{estimateder}
\end{align}
Since $\{\omega^j\}$ are orthonormal in $L^2$ we can write
$$
|a(t)|^2 = (u^n, u^n) \leq \frac{1}{\rho_*}\int_{\mathbb{T}^3}\rho^n(t, x)|u^n(t,x)|^2\,\mathrm{d}x
$$
Multiplying \eqref{approximate} by $a(t)$ and summing over $r$ we get
\begin{multline}\label{somethingdumbthing}
    \int_{\mathbb{T}^3}\rho^n(t, x)\p_tu^n(t, x)\cdot u^n(t, x)\,\mathrm{d}x + \int_{\mathbb{T}^3}\rho^n(t, x)\,[\overline{u}^n(t, x)\nabla_xu^n(t, x)]\cdot u^n(t, x)\,\mathrm{d}x\\
    + \int_{\mathbb{T}^3}\stress^n_{\alpha}(t, x) : \D u^n(t, x)\,\mathrm{d}x = 0
\end{multline}
Using \eqref{rhobar} we get
\begin{align*}
    &\int_{\mathbb{T}^3}\rho^n(t, x)\,[\overline{u}^n(t, x)\nabla_xu^n(t, x)]\cdot u^n(t, x)\,\mathrm{d}x = \frac{1}{2}\int_{\mathbb{T}^3}\rho^n(t, x)\overline{u}^n(t, x)\nabla_x(|u^n|^2)\,\mathrm{d}x\\
    &= -\frac{1}{2}\int_{\mathbb{T}^3}\DIV_x(\rho^n\overline{u}^n)|u^n(t,x)|^2\,\mathrm{d}x = \frac{1}{2}\int_{\mathbb{T}^3}\p_t\rho^n(t, x)|u^n(t,x)|^2\,\mathrm{d}x
\end{align*}
And by the fact that
$$
\stress^n_{\alpha}\in\p F_{\alpha}(\D u^n)
$$
we have
$$
0\leq \int_{\mathbb{T}^3}F_{\alpha}(\D u^n) + F^{*}_{\alpha}(\stress^n_{\alpha})\,\mathrm{d}x = \int_{\mathbb{T}^3}\stress^n_{\alpha}(t,x) : \D u^n(t,x)\,\mathrm{d}x
$$
Thus we finally arrive at
\begin{multline}
    \frac{1}{2}\frac{d}{dt}\int_{\mathbb{T}^3}\rho^n(t,x)|u^n(t,x)|^2\,\mathrm{d}x\leq \frac{1}{2}\frac{d}{dt}\int_{\mathbb{T}^3}\rho^n(t,x)|u^n(t,x)|^2\,\mathrm{d}x\\ + \int_{\mathbb{T}^3}F_{\alpha}(\D u^n) + F^{*}_{\alpha}(\stress^n_{\alpha})\,\mathrm{d}x = 0
\end{multline}
Which means that
$$
\frac{1}{\rho_*}\int_{\mathbb{T}^3}\rho^n(t,x)|u^n(t,x)|^2\,\mathrm{d}x\leq \frac{1}{\rho_*}\int_{\mathbb{T}^3}\rho^n_0(x)|P^nu_0(x)|^2\,\mathrm{d}x
$$
and we just need to take $L^2$ greater than the expression on the right hand side of the inequality and acquire \eqref{estimatea}. To get \eqref{estimateder} let us notice that similarly as before
$$
\left|\frac{d}{dt}a\right|^2 = \left(\frac{d}{dt}u^n, \frac{d}{dt}u^n\right)\leq\frac{1}{\rho_*}\int_{\mathbb{T}^3}\rho^n(t,x)|\p_t u^n(t,x)|^2\,\mathrm{d}x
$$
Now we multiply \eqref{approximate} by $\frac{d}{dt}a$ to get
\begin{multline*}
    \int_{\mathbb{T}^3}\rho^n(t,x)|\p_t u^n(t,x)|^2\,\mathrm{d}x + \int_{\mathbb{T}^3}\rho^n(t, x)\,[\overline{u}^n(t, x)\nabla_xu^n(t, x)]\cdot \p_tu^n(t, x)\,\mathrm{d}x\\
    + \int_{\mathbb{T}^3}\stress^n_{\alpha}(t, x) : \D (\p_tu^n(t, x))\,\mathrm{d}x = 0
\end{multline*}
Thanks to the \eqref{estimatea} we already know that $\sqrt{\rho^n}\,\overline{u}^n\,\nabla_xu^n$ is a bounded function, thus by the Young's inequality
\begin{align*}
    &\int_{\mathbb{T}^3}\left|\rho^n(t, x)\,[\overline{u}^n(t, x)\nabla_xu^n(t, x)]\cdot \p_tu^n(t, x)\right|\,\mathrm{d}x \leq \frac{\varepsilon}{2}\int_{\mathbb{T}^3}\rho^n(t,x)|\p_t u^n(t,x)|^2\,\mathrm{d}x\\
    &+ \frac{1}{2\varepsilon}\int_{\mathbb{T}^3}\rho^n(t,x)\,|\overline{u}^n(t, x)\nabla_xu^n(t, x)|^2\,\mathrm{d}x \leq \varepsilon\int_{\mathbb{T}^3}\rho^n(t,x)|\p_t u^n(t,x)|^2\,\mathrm{d}x + A
\end{align*}
where $\varepsilon >0$ is some arbitrary, small number and $A >0$ is some constant. Similarly we can estimate
\begin{align*}
    &\int_{\mathbb{T}^3}|\stress^n_{\alpha}(t,x):\D(\p_t u^n(t,x)|\,\mathrm{d}x\leq |\mathbb{T}^3|\,\|\stress^n_{\alpha}\|_{\infty}\max_{k=1, ..., n}|\D\omega^k(x)|\,\left|\frac{d}{dt}a\right|\\
    & \leq R + \rho_*\varepsilon\left|\frac{d}{dt}a\right|^2 \leq R + \varepsilon\int_{\mathbb{T}^3}\rho^n(t,x)|\p_tu^n(t,x)|^2\,\mathrm{d}x
\end{align*}
for some constant $R > 0$ and arbitrarly small $\varepsilon > 0$. Applying those of the inequalities above we arrive at
\begin{align}
    \int_{\mathbb{T}^3}\rho^n(t,x)|\p_tu^n(t,x)|^2\,\mathrm{d}x \leq \frac{A + R}{1 - 2\varepsilon} = K
\end{align}
Due to both \eqref{estimatea} and \eqref{estimateder} we deduce that the function defined in \eqref{functionschauder} indeed fulfills the Schauder fixed point theorem. It means that we have proven the existence of the solution of the system \eqref{galerkin} and \eqref{galerkin2} on some interval $(0, t^*)$. Let $t_0$ be a maximal time for which the solution can be achieved. We would like to show that $t_0 = T$. Assume $t_0 < T$. Then since
$$
\|a\|_{\infty, (0, t_0)} \leq L
$$
we have 
$$
a(t_0) = \lim_{t\to t_0^-}a(t) \leq L
$$
As all of the estimates were global for us, we could extend the same line of thinking to some interval $(t_0, t_0 + \delta)$ which is a contradiction. Let us note that by multiplying \eqref{galerkin} by $\rho^{\gamma - 1}$ and integrating over time and space we get
\begin{align}
    \int_{\mathbb{T}^3}\rho^{n\, \gamma}(\tau, x)\,\mathrm{d}x = \int_{\mathbb{T}^3}\rho^{n\,\gamma}_0(x)\,\mathrm{d}x
\end{align}
and the equality \eqref{somethingdumbthing} implies the energy equality
\begin{multline}\label{energyalpha}
\frac{1}{2}\int_{\mathbb{T}^3}\rho^n(\tau,x)|u^n(\tau,x)|^2\,\mathrm{d}x + \int_0^\tau\int_{\mathbb{T}^3}F_{\alpha}(\D u^n) + F^{*}_{\alpha}(\stress^n_{\alpha})\,\mathrm{d}x\mathrm{d}t\\
= \frac{1}{2}\int_{\mathbb{T}^3}\rho^n_0(x)|P^nu_0(x)|^2\,\mathrm{d}x
\end{multline}
Now we would like to pass to the limit with $\alpha \to 0$. To do this we need to derive some kind of compactness of the $\{\stress^n_{\alpha}\}_{\alpha}$ set. Thanks to the \eqref{energyalpha} we can notice that
\begin{align}
    \int_0^T\int_{\mathbb{T}^3}F^*(\stress^n_{\alpha})\,\mathrm{d}x\mathrm{d}t \leq \int_0^T\int_{\mathbb{T}^3}F^*_{\alpha}(\stress^n_{\alpha})\,\mathrm{d}x\mathrm{d}t \leq \frac{1}{2}\int_{\mathbb{T}^3}\rho^n_0(x)|P^nu_0(x)|^2\,\mathrm{d}x
\end{align}
and due to \eqref{suplin} and the de la Vall\'{e}e Poussin theorem we can deduce the existence of the subsequence such that
\begin{align}\label{weakstressalpha}
    \stress^n_{\alpha} \rightharpoonup \stress^n \text{ in }L^1((0, T) \times \mathbb{T}^3; \symmetric)
\end{align}
First, let us show that indeed
\begin{align}\label{subdiffn}
\stress^n\in\p F(\D u^n)
\end{align}
As
$$
\stress^n_{\alpha}\in F_{\alpha}(\D u^n)
$$
we know that for any $(t, x), D$ the inequality
\begin{align}
    0\leq F_{\alpha}(D) - \stress^n_{\alpha}:(D - \D u^n) - F_{\alpha}(\D u^n)
\end{align}
holds. Multiplying this inequality by the characteristic function $\mathbf{1}_A$ of any measurable set $A$ and integrating over space and time we get
\begin{align}
    0 \leq \int_0^T\int_{\mathbb{T}^3}[F_{\alpha}(D) - \stress^n_{\alpha}:(D - \D u^n) - F_{\alpha}(\D u^n)]\mathbf{1}_A\,\mathrm{d}x\,\mathrm{d}t
\end{align}
Since $\D u^n$ is a continuous function we can use \eqref{uniformconv} which implies $F_{\alpha} \to F$ uniformly in $(t, x)$. Fore-mentioned observation together with \eqref{weakstressalpha} implies
\begin{align}
    0 \leq \int_0^T\int_{\mathbb{T}^3}[F(D) - \stress^n:(D - \D u^n) - F(\D u^n)]\mathbf{1}_A\,\mathrm{d}x\,\mathrm{d}t
\end{align}
Since $A$ is arbitrary we get
$$
0\leq F(D) - \stress^n:(D - \D u^n) - F(\D u^n)
$$
and in consequence \eqref{subdiffn}. We can use the similar procedure (i. e. multiply the equation by the arbitrary characteristic function, converge weakly in $\alpha$ and use the arbitrariness of the characteristic function) to derive from \eqref{galerkin2}
\begin{multline}
    \int_{\mathbb{T}^3}\rho^n(t, x)\p_tu^n(t, x)\cdot\omega^r(x)\,\mathrm{d}x + \int_{\mathbb{T}^3}\rho^n(t, x)\,[u^n(t, x)\nabla_xu^n(t, x)]\cdot\omega^r(x)\,\mathrm{d}x\\
    + \int_{\mathbb{T}^3}\stress^n(t, x) : \D\omega^r(x)\,\mathrm{d}x = 0
\end{multline}
and due to
\begin{multline}
\int_{\mathbb{T}^3}F_{\alpha}(\D u^n) + F^{*}_{\alpha}(\stress^n_{\alpha})\,\mathrm{d}x = \int_{\mathbb{T}^3}\stress^n_{\alpha}(t,x) : \D u^n(t,x)\,\mathrm{d}x\\
\longrightarrow \int_{\mathbb{T}^3}\stress^n(t,x) : \D u^n(t,x)\,\mathrm{d}x = \int_{\mathbb{T}^3}F(\D u^n) + F^{*}(\stress^n)\,\mathrm{d}x
\end{multline}
we can imply the energy equality
\begin{multline}\label{energyn}
\frac{1}{2}\int_{\mathbb{T}^3}\rho^n(\tau, x)|u^n(\tau,x)|^2\,\mathrm{d}x + \int_0^\tau\int_{\mathbb{T}^3}F(\D u^n) + F^{*}(\stress^n)\,\mathrm{d}x\mathrm{d}t\\
= \frac{1}{2}\int_{\mathbb{T}^3}\rho^n_0(x)|P^nu_0(x)|^2\,\mathrm{d}x
\end{multline}

\subsection{Limits in $n$}

\noindent Now we can transition into limiting in $n$. Obviously all of the limits below are up to the subsequence, but since we derive finitely many of them we can find one fulfilling all of them.\\
\\
From the \eqref{energyn} as well as \eqref{densitybound} we derive
\begin{align}
u^n \wstar u \text{ in }L^{\infty}((0, T); L^2(\mathbb{T}^3;\mathbb{R}^3))
\end{align}
thanks to the Banach-Bourbaki-Alaoglu theorem. Similarly to the $\stress^n_{\alpha}$ limit we notice
\begin{align}
    \stress^n \rightharpoonup \stress \text{ in } L^1((0, T)\times\mathbb{T}^3;\symmetric)
\end{align}
and
\begin{align}
    \D u^n \rightharpoonup \D u \text{ in } L^1((0, T)\times\mathbb{T}^3;\symmetric)
\end{align}
As for the density sequence we can use the Banach-Bourbaki-Alaoglu theorem to infer
\begin{align}
    \rho^n \wstar \rho \text{ in } L^\infty((0, T) \times \mathbb{T}^3)
\end{align}
 and thanks to the renormalization property of the transport equation
\begin{align}
    \rho^n \rightarrow \rho \text{ in } C([0, T]; L^2(\mathbb{T}^3))
\end{align}
One may find the detailed information in \cite{wroblewska2013unsteady} Section 3.2, \cite{diperna1989ordinary} and most importantly Section 3 of \cite{capuzzo1996onsome}. Last thing we need to do is to deal with the expression $\rho^n\, u^n\otimes u^n$. We note that
$$
\rho^n\,(u^n\otimes u^n - u\otimes u)
$$
due to the energy equality \eqref{energyn} is bounded in $L^\infty((0, T); L^1(\mathbb{T}^3;\symmetric))$. Using the embedding $L^\infty((0, T); L^1(\mathbb{T}^3;\symmetric)) \hookrightarrow L^\infty((0, T); \mathcal{M}(\mathbb{T}^3;\symmetric))$ we arrive at
\begin{align}
    \rho^n\,(u^n\otimes u^n - u\otimes u) \wstar m \text{ in } L^\infty((0, T); \mathcal{M}(\mathbb{T}^3;\symmetric))
\end{align}
Employing those convergences we get momentum equation, incompressibility condition as well as the energy equality for the density
\begin{align}\label{energydenisty}
    \int_{\mathbb{T}^3}\rho(t,x)^\gamma\,\mathrm{d}x = \int_{\mathbb{T}^3}\rho_0(x)^\gamma\,\mathrm{d}x
\end{align}
Multiplying \eqref{energyn} by a non-negative function $\psi\in C^\infty_c([0, T))$ such that $\psi(0) = 1$ and $\p_t\psi \leq 0$ we get
\begin{multline}
-\int_0^T\p_t\psi\frac{1}{2}\int_{\mathbb{T}^3}\rho^n(t, x)|u^n(t,x)|^2\,\mathrm{d}x\mathrm{d}t + \int_0^T\psi\int_{\mathbb{T}^3}F(\D u^n) + F^{*}(\stress^n)\,\mathrm{d}x\mathrm{d}t\\
= \frac{1}{2}\int_{\mathbb{T}^3}\rho^n_0(x)|P^nu_0(x)|^2\,\mathrm{d}x
\end{multline}
Rewriting
$$
\frac{1}{2}\rho^n|u^n| = \frac{1}{2}\mathrm{tr}[\rho^n(u^n\otimes u^n - u\otimes u)] + \frac{1}{2}\rho^n|u|^2
$$
and using the known convergences
\begin{multline}
    -\int_0^T\p_t\psi\frac{1}{2}\int_{\mathbb{T}^3}\rho^n(t, x)|u^n(t,x)|^2\,\mathrm{d}x\mathrm{d}t\longrightarrow\\
    -\int_0^T\p_t\psi\frac{1}{2}\int_{\mathbb{T}^3}\rho(t, x)|u(t,x)|^2\,\mathrm{d}x\mathrm{d}t
    - \int_0^T\p_t\psi\, D(t)\,\mathrm{d}t
\end{multline}
where
$$
D(t) = \frac{1}{2}\mathrm{tr}(m_t(\mathbb{T}^3))
$$
Now since $F^*$ is l.s.c. and convex we have
$$
\int_0^T\psi\int_{\mathbb{T}^3}F^*(\stress)\,\mathrm{d}x\mathrm{d}t\leq \liminf_{n\to\infty}\int_0^T\psi\int_{\mathbb{T}^3}F^*(\stress^n)\,\mathrm{d}s\mathrm{d}t
$$
Similarly
$$
\int_0^T\psi\int_{\mathbb{T}^3}F(\D u)\,\mathrm{d}x\mathrm{d}t\leq \liminf_{n\to\infty}\int_0^T\psi\int_{\mathbb{T}^3}F(\D u^n)\,\mathrm{d}s\mathrm{d}t
$$
Thus due to the density argument we have an inequality
\begin{multline*}
    -\int_0^T\p_t\phi\frac{1}{2}\int_{\mathbb{T}^3}\rho(t, x)|u(t,x)|^2\,\mathrm{d}x\mathrm{d}t
    - \int_0^T\p_t\phi\, D(t)\,\mathrm{d}t\\
    + \int_0^T\phi\int_{\mathbb{T}^3}F(\D u) + F^{*}(\stress)\,\mathrm{d}x\mathrm{d}t \leq \frac{1}{2}\int_{\mathbb{T}^3}\rho_0(x)|u_0(x)|^2\,\mathrm{d}x
\end{multline*}
for $\phi$ being equal to $1$ on $[0, \tau]$, going linearly to 0 on $[\tau, \tau + \delta]$ and then $0$ on $[\tau +\delta, T]$. Using the Lebesgue's differentiation theorem and adding \eqref{energydenisty} we get the energy inequality.\\
\\
All that is left to prove is \eqref{trace}. To this end let us show the following lemma
\begin{lem}\label{tracebound}
Let $\mu\in \mathcal{M}(\mathbb{T}^3; \symmetric)$ such that
\begin{align}\label{positivesemi}
\int_{\mathbb{T}^3}\phi(x)\xi\otimes\xi\,\mathrm{d}\mu(x) \geq 0
\end{align}
for any $\phi\in C(\mathbb{T}^3)$, $\phi \geq 0$ and $\xi \in \mathbb{R}^3$, then
$$
|\mu|(\mathbb{T}^3) \lesssim \mathrm{trace}(\mu(\mathbb{T}^3))
$$
\end{lem}
\begin{proof}
The inequality \eqref{positivesemi} can be written as
\begin{align}
    \sum_{i, j = 1}^3 \int_{\mathbb{T}^3}\phi(x)\xi_i\xi_j\,\mathrm{d}\mu_{ij}(x) \geq 0
\end{align}
Substituting $\xi_i = 1$ and the rest of them equal to $0$ we arrive at
$$
\mu_{ii} \geq 0
$$
Similarly replacing $\xi_ 1 = 1$, $\xi_2 = 1$ and $\xi_3 = 0$ we have
$$
\int_{\mathbb{T}^3}\phi(x)\,\mathrm{d}\mu_{11}(x) + \int_{\mathbb{T}^3}\phi(x)\,\mathrm{d}\mu_{22}(x) + 2\int_{\mathbb{T}^3}\phi(x)\,\mathrm{d}\mu_{12}(x) \geq 0
$$
and with $\xi_1 = -1$, $\xi_2 =1$ and $\xi_3 = 0$
$$
\int_{\mathbb{T}^3}\phi(x)\,\mathrm{d}\mu_{11}(x) + \int_{\mathbb{T}^3}\phi(x)\,\mathrm{d}\mu_{22}(x) - 2\int_{\mathbb{T}^3}\phi(x)\,\mathrm{d}\mu_{12}(x) \geq 0
$$
Denote by $U$, $V$ the Hahn decomposition for $\mu_{12}$ (see for example \cite{rudin2006real}, Theorem 6.14, page 125-126) and let $\phi\in C(\mathbb{T}^3)$, $0\leq \phi\leq 1$ be such that $\|\mathbf{1}_U - \phi\|_1 \leq \varepsilon$, then
\begin{multline*}
\mu^+_{12}(\mathbb{T}^3) = \mu_{12}(U) = \int_{\mathbb{T}^3}\mathbf{1}_{U} - \phi(x)\,\mathrm{d}\mu_{12}(x) + \int_{\mathbb{T}^3}\phi(x)\,\mathrm{d}_{12}(x)\\
\leq \frac{1}{2}\left(\int_{\mathbb{T}^3}\phi(x)\,\mathrm{d}\mu_{11}(x) + \int_{\mathbb{T}^3}\phi(x)\,\mathrm{d}\mu_{22}(x)\right) + \varepsilon \leq \frac{1}{2}\mu_{11}(\mathbb{T}^3) + \frac{1}{2}\mu_{22}(\mathbb{T}^3) + \varepsilon
\end{multline*}
Similarly for $V$ (with the choice of a different $\phi$)
\begin{multline*}
\mu^-_{12}(\mathbb{T}^3) = -\mu_{12}(V) = \int_{\mathbb{T}^3}\phi(x) - \mathbf{1}_V\,\mathrm{d}\mu_{12}(x) - \int_{\mathbb{T}^3}\phi(x)\,\mathrm{d}_{12}(x)\\
\leq \frac{1}{2}\left(\int_{\mathbb{T}^3}\phi(x)\,\mathrm{d}\mu_{11}(x) + \int_{\mathbb{T}^3}\phi(x)\,\mathrm{d}\mu_{22}(x)\right) + \varepsilon \leq \frac{1}{2}\mu_{11}(\mathbb{T}^3) + \frac{1}{2}\mu_{22}(\mathbb{T}^3) + \varepsilon
\end{multline*}
From the arbitrariness of $\varepsilon$
$$
|\mu_{12}|(\mathbb{T}^3) \leq \mu_{11}(\mathbb{T}^3) + \mu_{22}(\mathbb{T}^3)
$$
With the same argument the analogous inequalities can be achieved for $|\mu_{13}|$ and $|\mu_{23}|$ therefore
$$
|\mu|(\mathbb{T}^3) \leq 4\,\mathrm{trace}(\mu(\mathbb{T}^3))
$$
\end{proof}
\noindent Hence to prove \eqref{trace} we need to show \eqref{positivesemi} for $m_\tau$. The proof is similar to \cite{abbatiello2020onaclass} Section 3, but let us repeat the argument given there. We have
\begin{multline*}
\int_0^T\int_{\mathbb{T}^3}\psi(t)\phi(x)\xi\otimes\xi\,\mathrm{d}m = \lim_{n\to\infty}\int_0^T\int_{\mathbb{T}^3}\psi(t)\phi(x)\xi\otimes\xi:\rho_n(u_n\otimes u_n - u\otimes u)\,\mathrm{d}x\mathrm{d}t\\
\geq \liminf_{n\to\infty}\int_0^T\int_{\mathbb{T}^3}\psi(t)\phi(x)\xi\otimes\xi:\rho_n(u_n\otimes u_n)\,\mathrm{d}x\mathrm{d}t - \int_0^T\int_{\mathbb{T}^3}\psi(t)\phi(x)\xi\otimes\xi:\rho(u\otimes u)\,\mathrm{d}x\mathrm{d}t
\end{multline*}
for any $\psi\in C((0, T))$, $\phi\in C(\mathbb{T}^3)$ and $\phi, \psi \geq 0$. We may now write
$$
\xi\otimes\xi : \rho_n(u_n\otimes u_n) = |\xi\,\sqrt{\rho_n}u_n|^2
$$
and from the weak lower semi-continuity (for example in $L^2$) of a convex, non-negative function $z \mapsto |\xi\, z|^2$ it follows that
$$
\liminf_{n\to\infty}\int_0^T\int_{\mathbb{T}^3}\psi(t)\phi(x)\xi\otimes\xi:\rho_n(u_n\otimes u_n)\,\mathrm{d}x\mathrm{d}t \geq \int_0^T\int_{\mathbb{T}^3}\psi(t)\phi(x)\xi\otimes\xi:\rho(u\otimes u)\,\mathrm{d}x\mathrm{d}t
$$
Thus
$$
\int_{\mathbb{T}^3}\phi(x)\xi\otimes\xi\,\mathrm{d}m_\tau\geq 0
$$
for almost every $\tau$, which is what we wanted to prove.\qed

\section{Mv-strong uniqueness}

\begin{thm}
Let the pair $(\rho, u)$, together with $\stress$ and $m$ be a dissipative measure-valued solution with initial datum $(\rho_0, u_0)$. Then if $P\in C^1([0, T]\times\mathbb{T}^3)$ and $U\in C^1([0, T]\times\mathbb{T}^3; \R^3)$, together with $\widehat{\stress}\in C^1([0, T]\times\mathbb{T}^3; \symmetric)$ is a strong solution with the same initial conditions, then $u = U$, $\rho = P$, $\stress = \widehat{\stress}$ and $m\equiv 0$.
\end{thm}
\begin{proof}
In order to prove the theorem above we will define the relative energy and use Gr\"{o}nwall's inequality to prove that it's equal to $0$. We define
\begin{multline}
    E_{\text{rel}}(\tau) = \frac{1}{2}\int_{\mathbb{T}^3}\rho(\tau,x)|u - U|^2(\tau, x)\,\mathrm{d}x\\
    + \int_{\mathbb{T}^3}\frac{1}{\gamma}\rho(\tau, x)^\gamma - P(\tau, x)^{\gamma -1}\rho(\tau, x) + \frac{\gamma-1}{\gamma}P(\tau, x)^\gamma\,\mathrm{d}x + D(\tau)
\end{multline}
To get the appropriate inequality let us work backwards from the energy inequality \eqref{energyinequality}. First, let us notice that by substituting $U$ in \eqref{momentumequation} we arrive at
\begin{multline}\label{momentumstrong}
        \int_0^\tau\int_{\mathbb{T}^3}\rho u\cdot\p_tU + \rho u\otimes u:\nabla_xU\,\mathrm{d}x\mathrm{d}t - \int_0^\tau\int_{\mathbb{T}^3}\stress:\D U\,\mathrm{d}x\mathrm{d}t + \int_0^\tau\int_{\mathbb{T}^3}\nabla_xU\,\mathrm{d}m_t\mathrm{d}t\\
        = \int_{\mathbb{T}^3}\rho u\cdot U(\tau, x) - \rho_0 u_0\cdot U(0, x)\,\mathrm{d}x
\end{multline}
And by the same operation on \eqref{something}, but with $\phi = \frac{1}{2}|U|^2$
\begin{align}\label{somethingstrong}
        \int_0^\tau\int_{\mathbb{T}^3}\rho\p_t\frac{1}{2}|U|^2 + \rho u\cdot\nabla_x\frac{1}{2}|U|^2\,\mathrm{d}x\mathrm{d}t = \int_{\mathbb{T}^3}\rho\frac{1}{2}|U|^2(\tau, x) - \rho_0\frac{1}{2}|U|^2(0, x)\,\mathrm{d}x
\end{align}
Using \eqref{momentumstrong}, \eqref{somethingstrong} together with the energy inequality \eqref{energyinequality} we have 
\begin{align*}
    &\frac{1}{2}\int_{\mathbb{T}^3}\rho(\tau,x)|u - U|^2(\tau, x)\,\mathrm{d}x + \frac{1}{\gamma}\int_{\mathbb{T}^3}\rho(\tau, x)^\gamma\,\mathrm{d}x + D(\tau) + \int_0^\tau\int_{\mathbb{T}^3}F(\D u) + F^*(\stress)\,\mathrm{d}x\mathrm{d}t\\
    &\leq -\int_0^\tau\int_{\mathbb{T}^3}[\rho u\cdot\p_tU + \rho u\otimes u:\nabla_xU]\,\mathrm{d}x\mathrm{d}t + \int_0^\tau\int_{\mathbb{T}^3}\stress:\D U\,\mathrm{d}x\mathrm{d}t - \int_0^\tau\int_{\mathbb{T}^3}\nabla_xU\,\mathrm{d}m_t\mathrm{d}t\\
    & + \int_0^\tau\int_{\mathbb{T}^3}\rho\p_tU\cdot U + \rho u\cdot (U\cdot\nabla_xU)\,\mathrm{d}x\mathrm{d}t + \frac{1}{\gamma}\int_{\mathbb{T}^3}\rho_0(x)^\gamma\,\mathrm{d}x
\end{align*}
Similarly let us substitute $\phi = P^{\gamma - 1}$ in the \eqref{something}
\begin{align}\label{somethingdensity}
    \int_0^\tau\int_{\mathbb{T}^3}\rho\p_tP^{\gamma -1} + \rho u\cdot\nabla_xP^{\gamma - 1}\,\mathrm{d}x\mathrm{d}t = \int_{\mathbb{T}^3}\rho P^{\gamma - 1}(\tau, x) - \rho_0 P^{\gamma - 1}(0, x)\,\mathrm{d}x
\end{align}
Now after applying \eqref{somethingdensity} as well as the equality
$$
\int_{\mathbb{T}^3}P^\gamma(\tau, x)\,\mathrm{d}x = \int_{\mathbb{T}^3}\rho_0^\gamma(x)\,\mathrm{d}x
$$
we get
\begin{align*}
    &E_{\text{rel}}(\tau) + \int_0^\tau\int_{\mathbb{T}^3}F(\D u) + F^*(\stress)\,\mathrm{d}x\mathrm{d}t \leq -\int_0^\tau\int_{\mathbb{T}^3}[\rho u\cdot\p_tU + \rho u\otimes u:\nabla_xU]\,\mathrm{d}x\mathrm{d}t\\
    &+ \int_0^\tau\int_{\mathbb{T}^3}\stress:\D U\,\mathrm{d}x\mathrm{d}t - \int_0^\tau\int_{\mathbb{T}^3}\nabla_xU\,\mathrm{d}m_t\mathrm{d}t + \int_0^\tau\int_{\mathbb{T}^3}\rho\p_tU\cdot U + \rho u\cdot (U\cdot\nabla_xU)\,\mathrm{d}x\mathrm{d}t\\
    & - \int_0^\tau\int_{\mathbb{T}^3}[\rho\p_tP^{\gamma - 1} + \rho u\cdot\nabla_xP^{\gamma -1}]\,\mathrm{d}x\mathrm{d}t
\end{align*}
Let us handle the last term. We have
\begin{align*}
    &- \int_0^\tau\int_{\mathbb{T}^3}[\rho\p_tP^{\gamma - 1} + \rho u\cdot\nabla_xP^{\gamma -1}]\,\mathrm{d}x\mathrm{d}t
     = - \int_0^\tau\int_{\mathbb{T}^3}\rho(\p_tP^{\gamma - 1} + U\cdot\nabla_xP^{\gamma-1})\,\mathrm{d}x\,\mathrm{d}t\\
     & + \int_0^\tau\int_{\mathbb{T}^3}\rho(U - u)\cdot\nabla_xP^{\gamma - 1}\,\mathrm{d}x\mathrm{d}t = -\int_0^\tau\int_{\mathbb{T}^3}\rho(\p_tP^{\gamma -1} + U\cdot\nabla_xP^{\gamma - 1})\,\mathrm{d}x\,\mathrm{d}t\\
     &+ \int_0^\tau\int_{\mathbb{T}^3}P(U -u)\cdot\nabla_xP^{\gamma - 1}\,\mathrm{d}x\mathrm{d}t + \int_0^\tau\int_{\mathbb{T}^3}(\rho - P)(U-u)\cdot\nabla_xP^{\gamma - 1}\,\mathrm{d}x\mathrm{d}t
\end{align*}
By the virtue of \eqref{equations} we have
\begin{align*}
    (\p_tP^{\gamma - 1} + U\cdot\nabla_xP^{\gamma - 1}) = (\gamma - 1)P^{\gamma - 2} (\p_tP + U\cdot\nabla_xP) = 0
\end{align*}
and by \eqref{incompressibility}
$$
\int_0^\tau\int_{\mathbb{T}^3}P(U - u)\cdot\nabla_xP^{\gamma - 1}\,\mathrm{d}x\mathrm{d}t = \frac{\gamma - 1}{\gamma}\int_0^\tau\int_{\mathbb{T}^3}(U - u)\cdot\nabla_xP^{\gamma}\,\mathrm{d}x\mathrm{d}t = 0
$$
Using the inequality between geometric mean and root mean square we get
\begin{align}
\left|\int_0^\tau\int_{\mathbb{T}^3}(\rho - P)(U-u)\cdot\nabla_xP\,\mathrm{d}x\mathrm{d}t\right| \leq C\int_0^\tau \int_{\mathbb{T}^3}|\rho - P|^2 + |U - u|^2\,\mathrm{d}x\mathrm{d}t
\end{align}
for some constant $C >0$. Since we work on a compact set for densities it is easy to check that there exists a constant $B>0$ such that
$$
(\rho - P)^2 \leq B\left(\frac{1}{\gamma}\rho^\gamma - P^{\gamma -1}\rho + \frac{\gamma-1}{\gamma}P^\gamma\right)
$$
Hence
\begin{align}\label{geaomean}
    C\int_0^\tau \int_{\mathbb{T}^3}|\rho - P|^2 + |U - u|^2\,\mathrm{d}x\mathrm{d}t \lesssim \int_0^\tau E_{\text{rel}}(t)\,\mathrm{d}t
\end{align}
\noindent To treat the first terms we write
\begin{align*}
    &-\int_0^\tau\int_{\mathbb{T}^3}[\rho u\cdot\p_tU + \rho u\otimes u:\nabla_xU]\,\mathrm{d}x\mathrm{d}t +  + \int_0^\tau\int_{\mathbb{T}^3}\rho\p_tU\cdot U + \rho u\cdot (U\cdot\nabla_xU)\,\mathrm{d}x\mathrm{d}t\\
    &= \int_0^\tau\int_{\mathbb{T}^3}\rho(U-u)\cdot\p_tU + \rho(U-u)\otimes u:\nabla_xU\,\mathrm{d}x\mathrm{d}t\\
    &= \int_0^\tau\int_{\mathbb{T}^3}\rho(U-u)\cdot\p_tU + \rho(U-u)\otimes U:\nabla_xU\,\mathrm{d}x\mathrm{d}t + \int_0^\tau\int_{\mathbb{T}^3}\rho(U-u)\otimes(u-U):\nabla_xU\,\mathrm{d}x\mathrm{d}t\\
    &= \int_0^\tau\int_{\mathbb{T}^3}\rho(U-u)\cdot(\p_tU + \nabla_xU\cdot U)\,\mathrm{d}x\mathrm{d}t + \int_0^\tau\int_{\mathbb{T}^3}\rho(U-u)\otimes(u-U):\nabla_xU\,\mathrm{d}x\mathrm{d}t
\end{align*}
With the use of \eqref{secondequation} we arrive at
\begin{align*}
    &\int_0^\tau\int_{\mathbb{T}^3}\rho(U-u)\cdot(\p_tU + \nabla_xU\cdot U)\,\mathrm{d}x\mathrm{d}t + \int_0^\tau\int_{\mathbb{T}^3}\rho(U-u)\otimes(u-U):\nabla_xU\,\mathrm{d}x\mathrm{d}t\\
    &= \int_0^\tau\int_{\mathbb{T}^3}\frac{\rho}{P}(U-u)\DIV(\widehat{\stress})\,\mathrm{d}x\mathrm{d}t + \int_0^\tau\int_{\mathbb{T}^3}\rho(U-u)\otimes(u-U):\nabla_xU\,\mathrm{d}x\mathrm{d}t
\end{align*}
For the second term we can see that
\begin{align*}
    \left|\int_0^\tau\int_{\mathbb{T}^3}\rho(U-u)\otimes(u-U):\nabla_xU\,\mathrm{d}x\mathrm{d}t\right| \leq \|\nabla_xU\|_\infty \cdot \int_0^\tau\int_{\mathbb{T}^3}\rho|u-U|^2\,\mathrm{d}x\mathrm{d}t
\end{align*}
and for the first one
\begin{align*}
    \int_0^\tau\int_{\mathbb{T}^3}\frac{\rho}{P}(U-u)\DIV(\widehat{\stress})\,\mathrm{d}x\mathrm{d}t = \int_0^\tau\int_{\mathbb{T}^3}\frac{\DIV\widehat{\stress}}{P}(U-u)(\rho - P)\,\mathrm{d}x\mathrm{d}t + \int_0^\tau\int_{\mathbb{T}^3}(U-u)\DIV(\widehat{\stress})\,\mathrm{d}x\mathrm{d}t
\end{align*}
Similarly as in \eqref{geaomean}
\begin{align}\label{second}
    \left|\int_0^\tau\int_{\mathbb{T}^3}\frac{\DIV\widehat{\stress}}{P}(U-u)(\rho - P)\,\mathrm{d}x\mathrm{d}t\right| \leq C\int_0^\tau E_{\text{rel}}(t)\,\mathrm{d}t
\end{align}
The last thing that is needed is to show
\begin{align}\label{third}
    \int_0^\tau\int_{\mathbb{T}^3}F(\D u)+F^*(\stress) - \stress : \D u\,\mathrm{d}x\mathrm{d}t + \int_0^\tau\int_{\mathbb{T}^3}(u-U)\DIV\widehat{\stress}\,\mathrm{d}x\mathrm{d}t \geq 0
\end{align}
A straight forward calculation gives us 
\begin{multline*}
     \int_0^\tau\int_{\mathbb{T}^3}F(\D u)+F^*(\stress) - \stress : \D u\,\mathrm{d}x\mathrm{d}t + \int_0^\tau\int_{\mathbb{T}^3}(u-U)\DIV\widehat{\stress}\,\mathrm{d}x\mathrm{d}t\\
     = \int_0^\tau\int_{\mathbb{T}^3}F(\D u) - \widehat{\stress}:(\D u - \D U) - F(\D U)\,\mathrm{d}x\mathrm{d}t + \int_0^\tau\int_{\mathbb{T}^3}F(\D U) + F^*(\stress) - \stress : \D U\,\mathrm{d}x\mathrm{d}t
\end{multline*}
By the virtue of Fenchel - Young inequality
\begin{align*}
     \int_0^\tau\int_{\mathbb{T}^3}F(\D U) + F^*(\stress) - \stress : \D U\,\mathrm{d}x\mathrm{d}t \geq 0
\end{align*}
and since $\widehat{\stress}\in\p F(\D U)$ we may deduce
\begin{align*}
    \int_0^\tau\int_{\mathbb{T}^3}F(\D u) - \widehat{\stress}:(\D u - \D U) - F(\D U)\,\mathrm{d}x\mathrm{d}t\geq 0
\end{align*}
Finally, combining \eqref{geaomean}, \eqref{second} and \eqref{third}, we arrive at
\begin{align*}
    E_{\text{rel}}(\tau) \leq C\int_0^\tau E_{\text{rel}}(t)\,\mathrm{d}t + \|\nabla_xU\|_{\infty} \int_0^\tau |m_t|(\mathbb{T}^3)\,\mathrm{d}t
\end{align*}
Which, by \eqref{trace}, implies
\begin{align*}
     E_{\text{rel}}(\tau) \leq C\int_0^\tau E_{\text{rel}}(t)\,\mathrm{d}t + C^{'}\|\nabla_xU\|_{\infty} \int_0^\tau D(t)\,\mathrm{d}t \lesssim \int_0^\tau E_{\text{rel}}(t)\,\mathrm{d}t
\end{align*}
Utilizing Gr\"{o}nwall's inequality we get $\rho = P$, $u = U$, $m\equiv 0$ and in the end $\stress = \widehat{\stress}$, what ends the proof.
\end{proof}

\end{document}